\documentclass{article}
\usepackage{ifdraft} 
\usepackage{cancel}
\usepackage[utf8]{inputenc}
\usepackage[T1]{fontenc}
\usepackage{lmodern,microtype,etex,amsmath,amsfonts,amssymb,amsthm,amscd,paralist,array,booktabs,setspace,tikz,mathrsfs,mathtools}
\usepackage[pass]{geometry}
\usepackage[shortlabels]{enumitem}

\newcommand{\COMMENT}[1]{\ifdraft{\begin{@empty}\color{olive}#1\end{@empty}}{}}

\usepackage[Coxeter]{dynkin-diagrams}
\tikzset{/Dynkin diagram/fold style/.style={white}}

\usepackage[hidelinks,pdftex]{hyperref}
\usepackage[nameinlink,capitalize]{cleveref}

\crefname{maintheorem}{Theorem}{Theorems}
\newtheorem{mainquestion}{Question}\crefname{mainquestion}{Question}{Questions}
\newtheorem{mainfact}{Fact}\crefname{mainfact}{Fact}{Facts}

\newtheorem{theorem}{Theorem}[section]

\newtheorem{prop}[theorem]{Proposition}
\newtheorem{observation}[theorem]{Observation}\crefname{observation}{Observation}{Observations}

\theoremstyle{definition}

\theoremstyle{remark}
\newtheorem{idea}{Idea}\crefname{idea}{Idea}{Ideas}

\author{Marcus Zibrowius}
\title{Dual involutions in finite Coxeter groups}

\begin{document}
\maketitle

\newcommand{\ZZ}{\mathbb{Z}}
\newcommand{\NN}{\mathbb{N}}
\newcommand{\RR}{\mathbb{R}}
\newcommand{\cardinality}[1]{\left|#1\right|}
\newcommand{\id}{\mathrm{id}}

\ifdraft{
  \newcommand{\explanation}[2]{\noindent\parbox[t]{\textwidth}{\parbox[t]{4em}{\(#1\)}\parbox[t]{\textwidth-3cm}{#2}}\\[6pt]}
}{
  \newcommand{\explanation}[2]{}
}
\newcommand{\simpleRoots}{\Sigma}
\explanation{\simpleRoots}{set of simple roots}
\newcommand{\folded}[1]{#1^\sigma}
\explanation{\folded{\simpleRoots}}{The root system obtained from \(\Sigma\) by folding along the automorphism \(\sigma\)}
\newcommand{\Weyl}{W}
\explanation{\Weyl}{Weyl group, grenerated by reflections in the simple roots}
\newcommand{\Weylparabolic}[1]{\Weyl_{#1}}
\explanation{\Weylparabolic{I}}{parabolic subgroup of \(\Weyl\) corresponding to \(I\subset \simpleRoots\)}
\newcommand{\longest}{{w_o}}
\explanation{\longest}{longest element of the Weyl group \(\Weyl\)}
\newcommand{\Weylo}{\Weyl_o}
\explanation{\Weylo}{subgroup of elements of Weyl group commuting with \(\longest\)}
\newcommand{\Weylfolded}{\Weyl^\sigma}
\explanation{\Weylfolded}{``folded'' subgroup of the Wely group; the subgroup fixed under conjugation with \(\sigma\)}
\newcommand{\longestparabolic}[1]{{w_{#1}}}
\explanation{\longestparabolic{I}}{longest element of \(\Weylparabolic{I}\) (often viewed as an element of \(\Weyl\))}
\newcommand{\standardinvolution}[1]{{c_{#1}}}
\explanation{\longestparabolic{I}}{standard involution associated with \(I\subset \simpleRoots\)}
\newcommand{\inv}[2]{c_{#1}^{#2}}
\explanation{\inv{k,l}{}}{certain involutions in Weyl groups that I define}
\newcommand{\innerproduct}[2]{\langle #1,#2\rangle}
\explanation{\innerproduct{\alpha}{\beta}}{inner product on ambient vector space}
\newcommand{\plusdim}{\dim^+}
\newcommand{\minusdim}{\dim^-}
\newcommand{\plusminusdim}{\dim^\pm}
\explanation{\minusdim(\tau)}{the dimension of \(\pm 1\)-Eigenspace of \(\tau\), for any involution \(\tau\)}


\begin{abstract}
  There is a well-known classification of conjugacy classes of involutions in finite Coxeter groups, in terms of subsets of nodes of their Coxeter graphs.
  In many cases, the product of an involution with the longest element \(\longest\) is again an involution.
  We identify the conjugacy class of this product involution in terms of said classification.
\end{abstract}

Consider a finite Coxeter group \(\Weyl\), and suppose first for simplicity that the longest element \(\longest\in\Weyl\) is central.  The question we will address is:  given an involution \(w\in \Weyl\), what is the conjugacy class of the involution \(\longest w\)?

To make this question precise, let us fix a finite Coxeter graph \(\simpleRoots\), and consider the
associated finite Coxeter group \(\Weyl(\simpleRoots)\). Recall that we can realize \(\Weyl(\simpleRoots)\) as the Euclidean reflection group determined by a root system\footnote{We consider root systems to be essential and reduced by definition, but not necessarily crystallographic.}
whose undirected Dynkin diagram is precisely the Coxeter graph \(\simpleRoots\).  Our initial assumption on \(\longest\) holds in the following cases:
\begin{mainfact}[{\cite[\S\,27--2]{kane}}]\label{fact:central}
  The longest element \(\longest\in\Weyl(\simpleRoots)\) is central if and only if it acts as \(-\id\) when \(\Weyl(\simpleRoots)\) is viewed as a Euclidean reflection group, if and only if
  \(\simpleRoots\) is a disjoint union of graphs of the types \(A_1\), \(BC_n\), \(D_{2m}\), \(E_7\), \(E_8\), \(F_4\), \(G_2(2m)\), \(H_3\) and \(H_4\).
\end{mainfact}

There is a well-known classification of conjugacy classes of involutions in \(\Weyl(\simpleRoots)\), as follows.
By construction, each node of \(\simpleRoots\) corresponds to a specific (generating) involution.
More generally, given any subgraph \(I\subset \simpleRoots\), we have an associated involution \(\longestparabolic{I}\in\Weyl\): the longest element of the parabolic subgroup \(\Weylparabolic{I}\) generated by the involutions corresponding to the nodes of \(I\).  Such an involution is called \emph{standard} if \(\Weylparabolic{I}\cong \Weyl(I)\) satisifes the equivalent conditions of \cref{fact:central}, and in this case we write \(\standardinvolution{I} := \longestparabolic{I}\).

\begin{mainfact}[{\cite[\S\,27--4]{kane}}]\label{fact:conjugacy}
  Any involution in \(\Weyl(\simpleRoots)\) is conjugate to a standard involution \(\standardinvolution{I}\).
\end{mainfact}
Moreover, two standard involutions \(\standardinvolution{I}\) and \(\standardinvolution{J}\) are conjugate if and only if \(I\) and \(J\) are \(\Weyl(\simpleRoots)\)-equivalent, a condition that is decidable using a graphical calculus \cite[\S\,28--5]{kane}.  The precise form of our question, then, is:
\begin{mainquestion}\label{question1}
  Suppose \(\Weyl(\simpleRoots)\) satisfies one of the equivalent conditions of \cref{fact:central}.
  Given a standard involution \(\standardinvolution{I}\), for which subset(s) \(J\subset \simpleRoots\) is \(\longest\standardinvolution{I} \sim \standardinvolution{J}\)? Here, \(\sim\) denotes conjugacy in \(\Weyl(\simpleRoots)\).
\end{mainquestion}

More generally, without any assumption on \(\longest\), we can consider the subgroup \(\Weylo(\simpleRoots)\subset \Weyl(\simpleRoots)\) of those elements that commute with \(\longest\).
In general, not every involution in \(\Weylo(\simpleRoots)\) is conjugate in \(\Weylo(\simpleRoots)\) to a standard involution \(\standardinvolution{I}\).  However, it is still true that every involution in \(\Weylo(\simpleRoots)\) is conjugate in \(\Weylo(\simpleRoots)\) to one of the involutions \(\longestparabolic{I}\) defined by an arbitrary subgraph \(I \subset \simpleRoots\).  Multiplication by \(\longest\) defines an involution on the set of conjugacy classes of involutions in \(\Weylo(\simpleRoots)\), and the question becomes:

\begin{mainquestion}\label{question2}
  Given a subgraph \(I\subset \simpleRoots\) such that \(\longestparabolic{I}\in\Weylo\), for which subgraph(s) \(J\subset \simpleRoots\) is \(\longest \longestparabolic{I} \sim \longestparabolic{J}\)?  Here, \(\sim\) denotes conjugacy in \(\Weylo(\simpleRoots)\).
\end{mainquestion}

We provide complete results to this question, tabulated at the end of the paper.

\subsubsection*{Overview}
Arguably, the most straight-forward path from the above questions to the below results would be via elementary brute force computations.
However, this seems both tedious and unenlightening.  We leave it as an exercise for the reader to verify the results that we have tabulated for the dihedral groups \(\Weyl(G_2(n))\) in this way, and also our results for two specific involutions in \(\Weyl(D_{2k})\) (see details at the end of \cref{sec:folding}).   All other results can be obtained more elegantly using the following ideas, spelled out in more detail in the body of the paper.

\begin{idea}\label{idea:eigenspaces}
  View \(\Weyl(\simpleRoots)\) a Euclidean reflection group as above, and consider the dimensions of the eigenspaces of the involutions.
\end{idea}
For a standard involution \(\standardinvolution{I}\), the dimension of the \(-1\)-eigenspace is equal to \(\cardinality{I}\), the number of nodes of \(I\), and the dimension of the \(+1\)-eigenspace is equal to \(\cardinality{\simpleRoots}-\cardinality{I}\).  Multiplication with \(\longest = -1\) interchanges these two eigenspaces.  So any subgraph \(J\) in \cref{question1} must have \(\cardinality{\simpleRoots}-\cardinality{I}\) nodes.

In many cases, this alone determines the conjugacy class of \(\longest \standardinvolution{I}\).
For example, \cref{question1} can be completely answered for \(\Weyl(H_3)\) and \(\Weyl(H_4)\) in this way.
In other cases, this yields at least  partial information.  For example, in \(\Weyl(BC_5)\) we have two non-conjugate standard involutions defined by subgraphs with a single vertex, and two standard involutions defined by subgraphs with four vertices:
\begin{align*}
  c_{\{1\}}       & = c_{\dynkin C{*oooo}} & c_{\{5\}}       & = c_{\dynkin C{oooo*}} \\
  c_{\{1,3,4,5\}} & = c_{\dynkin C{*o***}}  & c_{\{2,3,4,5\}}   & = c_{\dynkin C{o****}}
\end{align*}
Multiplication with \(\longest\) must send each conjugacy class of an involution in the first line to a conjugacy class of an involution in the second line, but further analysis is needed to resolve the ambiguity.

\begin{idea}[sub-root systems]\label{example:sub-root-systems}
  Refine \cref{idea:eigenspaces} by considering not just the dimensions of eigenspaces but the closed sub-root systems they contain.
\end{idea}
For example, consider \(\Weyl(BC_5)\) as the Euclidean reflection group determined by the root system \(C_5\).
The fixed-point space of \(\standardinvolution{\{5\}}\) contains the sub-root system generated by the simple roots $\alpha_1,\alpha_2,\alpha_3$, a closed sub-root system of type $A_3$.  So the \(-1\)-eigenspace of (any conjugate of) \(-\standardinvolution{\{5\}}\) must also contain a sub-root system of type $A_3$.
Using the Borel--de Siebenthal classification of closed sub-root systems, we can therefore rule out one of the two candidate representatives of \(-\standardinvolution{\{5\}}\) from above, and find that \(-\standardinvolution{\{5\}} \sim \standardinvolution{\{2,3,4,5\}}\), resolving the ambiguity.
Similar arguments are used in \cref{sec:E7E8} to resolve all ambiguities for \(\Weyl(E_7)\) and \(\Weyl(E_8)\).

\begin{idea}\label{idea:folding}
  Refine \cref{idea:eigenspaces} using „folding“.
\end{idea}
Following Steinberg \cite{steinberg:endomorphisms}, we can “fold” the root system of type \(A_{10}\) onto \(C_{5}\) and obtain an embedding \(\iota\colon \Weyl(BC_5)\hookrightarrow\Weyl(A_{10})\). This embedding sends the longest element of \(\Weyl(BC_5)\) to the longest element of \(\Weyl(A_{10})\), and induces a bijection between the set of conjucagy classes of involutions in \(\Weyl(BC_5)\) and the set of conjugacy classes of involutions in \(\Weylo(A_{10})\).  We can therefore consider, for each involution \(c\in \Weyl(BC_5)\), not just the dimensions of the eigenspaces of \(c\) but also the dimensions of the eigenspaces of \(\iota c\).  This yields an alternative proof of the above claim that \(-\standardinvolution{\{5\}} \sim \standardinvolution{\{2,3,4,5\}}\).  This approach is expanded in \cref{sec:folding} to complete the results for the classical families \(A_n\), \(BC_n\), \(D_n\), and for \(F_4\) and \(E_6\).

\subsubsection*{Acknowledgements}
The author has previously discussed some of these ideas on mathoverflow \cite{mathoverflow:collected}.
The results presented are motivated by previous results on the real K-theory of certain homogeneous spaces \cite{hz:exterior}, which they make more explicit.

\section{Calculations}
In the following, we upgrade the Coxeter graph \(\Sigma\) to a Dynkin diagram without changing notation.  For most finite Coxeter graphs, there is a unique Dynkin diagram that has the same underlying Coxeter graph, but for the Coxeter graph of type \(BC_n\), we can choose between the Dynkin diagrams \(B_n\) and \(C_n\).  We view \(\Weyl(\Sigma)\) as the finite Euclidean reflection group determined by a root system of Dynkin type \(\Sigma\).  Thus, \(\Weyl(\Sigma)\) is a subgroup of the isometry group of a Euclidean vector space \((V, \innerproduct{-}{-})\).  Each node of of the Dynkin diagram \(\Sigma\) corresponds to a simple root in \(V\), and we use the same notation \(\simpleRoots \subset V\) for this set of roots.  Each node of \(\simpleRoots\) also corresponds to a generator of \(\Weyl(\Sigma)\), namely to the reflection in the hyperplane orthogonal to the corresponding root.

\subsection{Dimensions of eigenspaces}
For an arbitrary element \(w\in\Weyl(\simpleRoots)\), we denote by \(\plusminusdim(w)\) the dimensions of the \(\pm1\)-eigenspaces of \(w\).
Recall that \(\longest\) maps the set of simple roots \(\simpleRoots\subset V\) to the set \(-\simpleRoots\subset V\), so that \(-\longest\) is an automorphism of the set \(\simpleRoots\), and more precisely an automorphism of the Dynkin diagram.  More generally, any subset of roots \(I\subset \simpleRoots\) is permuted by the involution \(-\longestparabolic{I}\).

\begin{prop}\label{lem:minusdim}
  The dimensions of eigenspaces of various involutions can be obtained by counting orbits, as follows.
  \begin{enumerate}[(1)]
  \item
    For any set of simple roots \(I\subset\simpleRoots\),
    \begin{flalign*}
      &\hspace{4em}\minusdim(\longestparabolic{I}) = \cardinality{\{\text{orbits of } -\longestparabolic{I} \text{ on } I\}}&
    \end{flalign*}
  \item
    For any set of simple roots \(I\subset\simpleRoots\) such that \(-\longest I = I\), we have \(\longestparabolic{I} \in \Weylo(\simpleRoots)\), and
    \begin{flalign*}
      &\hspace{4em}\minusdim(\longest \longestparabolic{I})
       = \cardinality{\{\text{non-trivial orbits of } \longest \longestparabolic{I} \text{ on } I \}} & \\
      &\hspace{10em} + \cardinality{\{\text{orbits of } -\longest \text{ on } \simpleRoots\setminus I\}}
    \end{flalign*}
  \item
    More generally, suppose \(\sigma\) is an involution isometry of \((V,\innerproduct{-}{-})\) such that \(\sigma(\simpleRoots) =  \simpleRoots\).  Assume further that \(\sigma(I) = I\), and that the restrictions of \(\sigma\) and \(-\longestparabolic{I}\) to \(I\) commute.  Then \(\sigma\) and \(-\longestparabolic{I}\) commute on all of \(V\), and
    \begin{flalign*}
      &\hspace{4em}\minusdim(\sigma(-\longestparabolic{I}))
      = \cardinality{\{\text{non-trivial orbits of } \sigma(-\longestparabolic{I}) \text{ on } I \}} & \\
      &\hspace{11em} + \cardinality{\{\text{orbits of } \sigma \text{ on } \simpleRoots\setminus I\}}
    \end{flalign*}
  \end{enumerate}
\end{prop}
\begin{proof}
  (1) Decompose the ambient vector space as \(V = \RR I \perp (\RR I)^\perp\).  (Explicitly, \(\RR I\) is the span of the simple roots \(\alpha \in I\), and  \((\RR I)^\perp\) is the span of the fundamental weights \(\omega_\beta\) with \(\beta \in \simpleRoots\setminus I\).)  Any element of \(\Weylparabolic{I}\) maps \(\RR I\) to \(\RR I\) and restricts to the identity on \((\RR I)^\perp\), so the claim follows.

  For (3), note that by assumption \(\sigma\) also maps \(\RR I\) to \(\RR I\), and hence, as an isometry, also maps \((\RR I)^\perp\) to \((\RR I)^\perp\).  The restrictions of \(\sigma\) and \(\longestparabolic{I}\) to \(\RR I\) commute by assumption, and their restrictions to \((\RR I)^\perp\) commute in any case.

  For (2), it suffices to check that \(-\longest\) satisfies all assumptions in (3).  Here, the main point is that the condition \(-\longest I = I\) already implies that \(-\longest\) and \(-\longestparabolic{I}\) commute on \(I\).  This is trivial in the cases when \(\longest = -\id\), listed in \cref{fact:central}.  In the remaining cases,
  this can easily be verified on a case-by-case basis.  Note that the action of \(-\longestparabolic{I}\) on a root \(\alpha\in I\) is determined by the connected component of \(I\) that contains \(\alpha\), and that \(-\longest\) maps this component isomorphically to the component containing \(-\longestparabolic{I}\alpha\).
\end{proof}

Part~(1) gives, in particular, the dimensions of eigenspaces of \(\longest = \longestparabolic{\simpleRoots}\).  As we have already noted in \cref{fact:central},  \(\longest\) is central if and only if \(\minusdim(\longest) = \dim V\), and in this case it is itself a standard involution. In the remaining cases, we find that \(\longest\) is conjugate to (one or more) standard involutions as follows:
\newcommand{\hidelastvertex}{\mathllap{\textcolor{white}{\rule{1em}{1ex}}}}
\begin{align*}
  \Weyl(A_n)&\colon \longest = \dynkin A{*****.x}\hidelastvertex \sim \dynkin A{*o*o*.x}\hidelastvertex = c_{\{1,3,5,\dots\}}\quad  (n\geq 2)\\
  \Weyl(D_{2m+1})&\colon \longest = \dynkin D{*****.***} \sim \dynkin D{o****.***} = c_{\{2,3,4,\dots,2m+1\}}\\ 
  \Weyl(E_6)& \colon \longest = \dynkin E{******} \sim \dynkin E{o****o} = c_{\{2,3,4,5\}}\\
  \Weyl(G_2(2m+1)) &  \colon \longest = \dynkin[gonality=k] G{**} \sim \dynkin[gonality={k}] G{*o} = c_{\{1\}} \sim \dynkin[gonality={k}] G{o*} = c_{\{2\}}  \quad (k := 2m+1)
\end{align*}
Indeed, in in each case, there is a unique conjugacy class of involutions \(\tau\) for which \(\minusdim(\tau)\) is maximal.

More generally, we can easily find a standard involution \(\standardinvolution{I}\) conjugate to the involution \(\longestparabolic{H}\) associated with any given subdiagram \(H\subset \simpleRoots\).  Indeed, using the above conjugacy relations, we can find a subdiagram \(I\subset H\) such that  \(\standardinvolution{I} \sim \longestparabolic{H}\) in \(\Weylparabolic{H}\), and then a fortiori \(\standardinvolution{I}\sim\longestparabolic{H}\) in \(\Weyl(\simpleRoots)\).

\subsection{Sub-root systems}
\label{sec:E7E8}
\label{sec:subroots}
For the Borel--de Siebenthal classification of closed sub-root systems, see for example \cite[\S\,12]{kane}. The only other observation we need is:

\begin{observation}\label{obs:subroots}
  The \(+1\)-eigenspaces of a standard involution \(\standardinvolution{I}\) contains a closed sub-root system whose Dynkin graph is the subgraph of \(\simpleRoots\) spanned by all nodes that are not connected by an edge to the nodes in \(I\).
  The \(-1\)-eigenspaces of a standard involution \(\standardinvolution{I}\) contains a closed sub-root system with Dynkin graph~\(I\).
\end{observation}
\begin{proof}
  The \(-1\)-eigenspace of \(\standardinvolution{I}\) is spanned by the roots in \(I\). The \(+1\)-eigenspace is spanned by the \emph{fundamental weights} corresponding to roots in \(\simpleRoots\setminus I\), and thus does not necessarily contain a sub-root system of type \(\simpleRoots\setminus I\). However, all nodes of \(\simpleRoots\) not connected to \(I\) correspond to simple roots that are orthogonal to the roots in \(I\), and hence lie in the fixed spaces of \(\standardinvolution{I}\).
\end{proof}
For example, for $E_7$, the dimensions of eigenspaces only leave one ambiguity, concerning the lines (4) and (5) in the results below.
By the above observation, the fixed space of the involution on the left in line (5) contains a sub-root system of type \(A_2\) (generated by the simple roots \(\alpha_1\) and \(\alpha_2)\).  Of the involutions on the right of lines (4) and (5), only the involution in line (5) has a \(-1\)-eigenspace that can contain a sub-root system of type \(A_2\).

Similary, for $E_8$, the dimensions of eigenspaces only leave one ambiguity, concerning the lines (5) and (6) in the results below.  By the observation above, both the \(+1\)-eigenspace and the \(-1\)-eigenspace of the involution in line (6) contains a sub-root system of type \(A_2\), but the \(-1\)-eigenspaces of the involutions in line (5) do not.

\subsection{Folding}
\label{sec:folding}
Consider an isometry \(\sigma\) of \((V,\innerproduct{-}{-})\) that maps \(\simpleRoots\) to \(\simpleRoots\).
We do \emph{not} assume that simple roots in the same orbit are necessarily orthogonal.
Steinberg shows that there is a ``folded''  root system \(\simpleRoots^\sigma\) in the fixed point space \(V^\sigma\), with one node for every \(\sigma\)-orbit in \(\simpleRoots\) \cite[1.32\,(b) and surrounding discussion]{steinberg:endomorphisms}.
It comes with a canonical surjection of Dynkin diagrams \(\pi\colon \simpleRoots \twoheadrightarrow \simpleRoots^\sigma\) and an associated inclusion of Weyl groups
\renewcommand{\Weylfolded}{\Weyl(\simpleRoots^\sigma)}
\begin{equation}\label{eq:folding-subgroup}
  \iota\colon \Weylfolded \hookrightarrow \Weyl(\simpleRoots)
\end{equation}
The image of \(\iota\) is precisely the subgroup of elements of \(\Weyl(\simpleRoots)\) that are invariant under conjugation by \(\sigma\).  In particular, for \(\sigma = -\longest\), we have \(\iota \Weylfolded = \Weylo(\simpleRoots)\). Given a subdiagram \(I\subset \simpleRoots\) invariant under \(\sigma\), we write \(\folded{I}\) for the associated ``folded'' subdiagram \(\pi(I)\subset \simpleRoots^\sigma\).

\begin{prop}\label{prop:folding-longest-words}
  Given a subset \(I\subset \simpleRoots\) such that \(\sigma(I) = I\), and such that the restrictions of \(\sigma\) and  \(-\longestparabolic{I}\) to \(I\) agree, we have \(\iota(\longestparabolic{\folded{I}}) = \longestparabolic{I}\).
\end{prop}

\begin{proof}
  Observe the following:
  \begin{compactitem}
  \item[(1)] \(\longestparabolic{I}\) is the unique element of \(\Weyl(\simpleRoots)\) satisfying \(\longestparabolic{I}(I) = -I\);\\
    \(\longestparabolic{\folded{I}}\) is the unique elmenent of \(\Weylfolded\) satisfying \(\longestparabolic{\folded{I}}(\folded{I}) = -\folded{I}\).
  \item[(2)] By \cite[1.32(a)]{steinberg:endomorphisms}, it suffices to show that \(\longestparabolic{I}\) agrees with \(\longestparabolic{\folded{I}}\) when restricted to \(V^\sigma\).
  \item [(3)] By assumption, \(\sigma\) and \(\longestparabolic{I}\) commute (compare proof of \cref{lem:minusdim}\,(3)).
  \end{compactitem}
  Combining (1) and (2), we see that it suffices to show that \(\longestparabolic{I}(\folded{I}) = -\folded{I}\).
  Given a simple root \(\alpha\in I\), let us write the \(\sigma\)-orbit of \(\alpha\) as \(\{\alpha, \sigma\alpha, \dots, \sigma^{l-1}\alpha\}\), for some minimal integer~\(l\).
  The associated simple root \(a_{[\alpha]} \in \folded{\simpleRoots}\) is of the form
  \[
    a_{[\alpha]} = r(\alpha + \sigma\alpha + \sigma^2 \alpha + \dots + \sigma^{l-1} \alpha)
  \]
  for some real \(r>0\) \cite[above Corollary~1.30]{steinberg:endomorphisms}.
  As \(-\longestparabolic{I}\) sends \(\sigma\)-orbits to \(\sigma\)-orbits, by (3), we find that
  \begin{align*}
    -\longestparabolic{I}(a_{[\alpha]})
    &= -\longestparabolic{I}(r\alpha + \sigma\alpha + \sigma^2\alpha + \dots + \sigma^{l-1} \alpha) \\
    &= r(-\longestparabolic{I}\alpha + \sigma(-\longestparabolic{I}\alpha) + \sigma^2(-\longestparabolic{I}\alpha) + \dots + \sigma^{l-1}(-\longestparabolic{I}\alpha))\\
    &= \tfrac{r}{s} a_{[-\longestparabolic{I}\alpha]}
  \end{align*}
  for some real \(s>0\).  As a (reduced) root system cannot contain proper multiples of its roots, it follows that
  \[
    -\longestparabolic{I}(a_{[\alpha]}) = a_{[-\longestparabolic{I}\alpha]}.
  \]
  As \(-\longestparabolic{I}(\alpha) \in I\), this shows that \(\longestparabolic{I}(a_{[\alpha]}) \in -\folded{I}\), as desired.
\end{proof}

The following proposition clarifies the relation between parabolic subgroups of \(\Weyl(\simpleRoots)\) and parabolic subgroups of \(\Weylfolded\):
\begin{prop}
  For any subset \(I\subset \simpleRoots\) such that \(\sigma(I) = I\), the embedding \(\iota\) in \eqref{eq:folding-subgroup} identifies the parabolic subgroup \(\Weylparabolic{\folded{I}}\) of \(\Weylfolded\) with  \(\Weylparabolic{I}\cap \iota\Weylfolded\).
\end{prop}
\begin{proof}
  The subgroup \(\Weylparabolic{\folded{I}}\) is generated by the reflections in the roots \(a_{[\alpha]}\) with \(\alpha\in I\).  By \cite[1.30\,(b)]{steinberg:endomorphisms}, these reflections map to the longest element of the parabolic subgroup of \(\Weyl\) corresponding to the \(\sigma\)-orbit of \(\alpha\) in \(I\).  (This is a special case of \cref{prop:folding-longest-words}.) So clearly \(\iota\Weylparabolic{\folded{I}}\subset \Weylparabolic{I} \cap \iota(\Weylfolded)\).

  Conversely, let us consider an arbitrary element \(w\in\Weylfolded\) such that \(\iota w \in \Weylparabolic{I}\). We need to show that \(w\) acts as the identity on any vector \(x\) in the orthogonal complement of \(\RR\{a_{[\alpha]} \mid [\alpha] \in \folded{I}\}\) in \(\folded{V}\).  By construction, the action of \(w\) on \(\folded{V}\) is the restriction to \(\folded{V}\) of the action of \(\iota w\) on \(V\).  Moreover, with notation as in the previous proof, we find that:
  \begin{align*}
    \innerproduct{x}{a_{[\alpha]}}
    &= r\textstyle\sum_{i=1}^l \innerproduct{x}{\sigma^i\alpha}\\
    &= rl \innerproduct{x}{\alpha} && \text{ as } x \in V^\sigma \\
  \end{align*}
  for some \(r > 0\).
  Thus, the assumption that \(x\in\folded{V}\) is orthogonal to \(a_{[\alpha]}\) is equivalent to the assumption that \(x\) is orthogonal to \(\alpha\).  Thus, \(\iota w \) acts trivially on \(x\).  Altogether, this shows that \(w.x = (\iota w). x = x\), as needed.
\end{proof}

\begin{proof}[Calculations for \(A_n\), \(BC_n\), \(D_{\mathrm{odd}}\), \(E_6\) and \(F_4\)]
  For \(\sigma= -\longest\), we have an isomorphism
  \[
    \iota\colon \Weylfolded \xrightarrow{\cong} \Weylo(\simpleRoots)\subset \Weyl(\simpleRoots).
  \]
  It induces a bijection:
  \[
    \frac{\{\text{involutions in \(\Weylfolded\)}\}}{\text{conjugation}} \xrightarrow{\quad 1:1\quad } \frac{\{\text{involutions in \(\Weylo(\simpleRoots)\)}\}}{\text{conjugation}}
  \]
  By \cref{prop:folding-longest-words}, multiplication with \(-1\) (= the longest element in the folded group \(\Weylfolded\)) on the set on the left corresponds to multiplication with \(\longest\) on the set on the right.  Consider the complete sets of representatives of conjugacy clases listed in \cref{sec:opposite-involutions-results}.  Using \cref{lem:minusdim}, we can easily compute the dimensions of the \(-1\)-eigenspaces of each representative \(w\), its corresponding representative \(\iota w\), and of \(-w\) and \(\iota(-w) = \longest \iota(w)\), see \cref{table:minusdims}.

  The involutions \(\inv{k,l}{}\) are chosen such that \(\inv{k,l}{}\mapsto\inv{k,l}{}\) under the following isomorphisms induced by folding along \(-\longest\):
  \begin{align*}
    \Weyl(BC_n)&\xrightarrow{\cong} \Weylo(A_{2n})\\
    \Weyl(BC_n)&\xrightarrow{\cong} \Weylo(A_{2n-1})\\
    \Weyl(BC_n)&\xrightarrow{\cong} \Weylo(D_{n+1}) \quad \text{ for odd \(n+1\)}
  \end{align*}
  By simultaneously considering the entries for \(BC_n\) and \(A_{2n}\) in the table, we easily see that \(-\inv{k,l}{} \sim \inv{k,n-2k-l}{}\)  in \(\Weyl(BC_n)\) and that \(\longest \inv{k,l}{} \sim \inv{k,n-2k-l}{}\) in \(\Weylo(A_{2n})\).  The results for \(\Weylo(A_{2n-1})\) and \(\Weyl(D_{\mathrm{odd}})\) immediately follow.

  For \(\Weyl(F_4)\) and \(\Weylo(E_6)\), the results are obtained by an analogous argument comparing the dimensions of \(-1\)-eigenspaces of involutions corresponding to each other under the folding isomorphism  \(\Weyl(F_4)\cong\Weylo(E_6)\).
\end{proof}

\begin{table}
  \centering
  \newcolumntype{C}{>{$}c<{$}}
  \newcolumntype{L}{>{$}l<{$}}
  \begin{tabular}{LLCCCL}
    \toprule
    & w                            & \minusdim(w) & \minusdim(\longest w) \\
    \midrule
    A_{2n}                    & \inv{k,l}{}                  & 2k+l         & n-l                   \\
    A_{2n-1}                  & \inv{k,l}{}                  & 2k+l         & n-l                   \\
    BC_n                      & \inv{k,l}{}                  & k+l          & n-k-l                 \\
    D_{n+1}, n+1 \text{ odd}  & \inv{k,l}{}, l \text{ even } & k+l          & n-k-l                 \\
    D_{n+1}, n+1 \text{ odd}  & \inv{k,l}{}, l \text{ odd }  & k+l+1        & n-k-l+1               \\
    \bottomrule
  \end{tabular}
  \caption{Dimensions of the \(-1\)-eigenspaces}\label{table:minusdims}
\end{table}

\begin{proof}[Callculations for \(D_{\mathrm{even}}\)]
In this case, \(\longest = -1\), so \(-\longest\) does not define a non-trivial folding.  But we still have an automorphism \(\sigma\) of order \(2\) that commutes with \(\longest\), and that defines a folding embedding  \(\Weyl(BC_n)\hookrightarrow \Weyl(D_{n+1})\).  In this case, the subgroup \(\Weylo(\simpleRoots)\) in the discussion above has to replaced by the subgroup \(\Weyl(\simpleRoots)^\sigma\) of elements fixed under conjugation by \(\sigma\). The arguments go through, because \(\longest\) commutes with \(\sigma\), and hence multiplication by \(\longest\) takes involutions in \(\Weyl(\simpleRoots)^\sigma\) to involultions in \(\Weyl(\simpleRoots)^\sigma\).
All chosen representatives of conjugacy classes of involutions are contained in \(\Weyl(\simpleRoots)^\sigma\), except for the two involutions \(\inv{-}{}\) and \(\inv{+}{}\).
These two exceptional involutions are not conjugate to any other standard involution, nor to each other.  As multiplication with \(\longest\) restricts to the set of \emph{other} conjugacy classes, it must also restrict to the two-element set consisting of these two conjugacy classes. The result presented below is obtained via an explicit computation.
\end{proof}

\newpage
\newgeometry{left=3.5cm,right=3.5cm,top=3cm,bottom=2cm}
\section{Results}
\label{sec:opposite-involutions-results}
\begin{@empty}
  \newcommand{\myrule}{\par\vspace{4pt}\rule{\linewidth}{0.4pt}\vspace{4pt}}
    \parskip0pt
    \parindent0pt
  \newcommand{\arrowlabel}{-1}
  \newcommand{\myarrow}{\quad\xleftrightarrow{\quad \arrowlabel \quad}\quad}
  \newcommand{\mycircarrow}{\quad{\rotatebox{90}{$\circlearrowleft$}}^{\arrowlabel}}
  \newcommand{\class    }[1]{\left\{  #1 \right\}}
  \newcommand{\classes  }[2]{\left\{ \substack{ #1                                                   \\ #2} \right\}}
  \newcommand{\classeses}[2]{\left\{ \substack{ #1                                                   \\ #2\\\\\dots} \right\}}
  \newcommand{\lineno}[1]{(\text{#1})}
  \newenvironment{result}[2]{%
    \par
    \noindent\parbox{3cm}{\framebox{\parbox{2cm}{\centering #1\\#2}}}
    \begin{minipage}{\linewidth-3cm}%
    }{%
    \end{minipage}
    \myrule
  }
  We indicate a subgraph \(I\subset \simpleRoots\) by colouring the nodes of \(I\) black.
  For the families of types \(A\), \(BC\) and \(D\), the longest element \(\longest\) acts on conjugacy classes of certain involutions \(c_{k,l}\) as follows:
  \begin{equation}\label{eq:common-involution-involution}
    \tag{$\dagger$}
    \renewcommand{\arrowlabel}{\longest}
    \class{\substack{\text{involutions conjugate to}\\\inv{k,l}{}} }  \myarrow \class{\substack{\text{involutions conjugate to}\\\inv{k,n-2k-l}{}} }
  \end{equation}
  In each case, the involutions \(\inv{k,l}{}\) are indexed by integers \(l \in \{0, \dots, n\}\) and \(k \in \{0, \dots, \left\lfloor \tfrac{n-l}{2} \right\rfloor\}\).

  \myrule
  \begin{result}{$\Weylo(A_{2n})$}{}
    Each involution is conjugate to precisely one of the involutions \(\inv{k,l}{}\), defined as follows:
    \begin{align*}
      \inv{k,l}{} : ~&
                       \begin{tikzpicture}
                           \dynkin[ply=2,baseline] A{*o*o.o*o*.oo.oo.***.**.***.oo.oo.*o*o.o*o*}
                           \dynkinBrace*[\text{two times } k \text{ marked roots}]{1}{8}
                           \dynkinBrace*[2l \text{ marked roots}]{13}{16}
                         \end{tikzpicture}
    \end{align*}
    The longest element \(\longest\) acts on conjugacy classes as in \eqref{eq:common-involution-involution} above.
  \end{result}
  \begin{result}{$\Weylo(A_{2n-1})$}{\(n\geq 2\)}
    Each involution is conjugate to precisely one of the involutions \(\inv{k,l}{}\), defined as follows:
    \begin{align*}
      \inv{k,l}{} : ~&
                       \begin{tikzpicture}
                         \dynkin[ply=2,baseline] A{*o*o.o*o*.oo.oo.***.***.***.oo.oo.*o*o.o*o*}
                         \dynkinBrace*[\text{two times } k \text{ marked roots}]{1}{8}
                         \dynkinBrace*[\min\{2l-1,0\} \text{ marked roots}]{13}{16}
                       \end{tikzpicture}
    \end{align*}
    The longest element \(\longest\) acts on conjugacy classes as in \eqref{eq:common-involution-involution} above.
  \end{result}
  \begin{result}{$\Weyl(BC_n)$}{}
    Each involution is conjugate to precisely one of the involutions \(\inv{k,l}{}\), defined as follows:
    \begin{align*}
      \inv{k,l}{} : ~&
                           \begin{tikzpicture}
                             \dynkin[arrows=false,baseline] B{*o*o.o*o*.oo.oo.**.***}
                             \dynkinBrace[k \text{ marked roots}]{1}{8}
                             \dynkinBrace[l \text{ marked roots}]{13}{17}
                           \end{tikzpicture}
    \end{align*}
    The element \(\longest = -1\) acts on conjugacy classes as in \eqref{eq:common-involution-involution} above.
  \end{result}

  \begin{result}{$\Weylo(D_{n+1})$}{($n+1$ odd)}
    Each involution is conjugate to precisely one of the involutions \(\inv{k,l}{}\), defined as follows:
    \begin{align*}
      \inv{k,l}{} : ~&
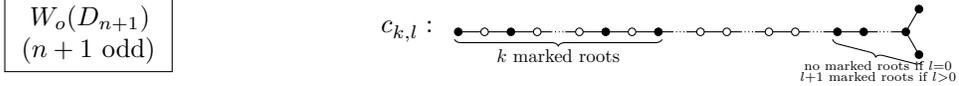
                       \begin{tikzpicture}
                         \dynkin[baseline] D{*o*o.o*o*.oo.oo.**.***}
                         \dynkinBrace[k \text{ marked roots}]{1}{8}
                         \dynkinBrace[\substack{\text{no marked roots if } l = 0\\l + 1\text{ marked roots if } l > 0}]{13}{17}
                       \end{tikzpicture}
    \end{align*}
    The longest element \(\longest\) acts on conjugacy classes as in \eqref{eq:common-involution-involution} above.
  \end{result}

  \begin{result}{$\Weyl(D_{n+1})$}{($n+1$ even)}
    Each involution is conjugate to at least one of the involutions \(\inv{k,l}{}\) defined as in the case when \(n+1\) is odd, or to precisely one of the following two additional involutions:
    \begin{align*}
      \inv{-}{}&\colon  \dynkin D{*o*o*o*.o*o*o*o}\\
      \inv{+}{}&\colon  \dynkin D{*o*o*o*.o*o*oo*}
    \end{align*}
    Here, \(\inv{k,l}{}\) is a standard involution if and only if \(l\) is odd. For even \(l\geq 2\), \(\inv{k,l}{}\sim \inv{k,l-1}{}\).
    The longest element  \(\longest = -1\) acts on the conjugacy classes of the involutions \(\inv{k,l}{}\) as in \eqref{eq:common-involution-involution} above.
    The conjugacy classes of \(\inv{-}{}\) and \(\inv{+}{}\) are fixed by multiplication with \(\longest\) when \(n+1\equiv 0 \mod 4\), and interchanged when \(n+1\equiv 2 \mod 4\).
  \end{result}

  \newpage
  \mbox{}\vfill
  \myrule
  \begin{result}{$\Weylo(E_6)$}{}
  \renewcommand{\arrowlabel}{\longest}
    \begin{flalign*}
      \class  {\dynkin[ply=2,backwards] E{oooooo}}                                   & \myarrow \class  {\dynkin[ply=2,backwards] E{******}} &  & \lineno{1} \\
      \classeses{\dynkin[ply=2,backwards] E{o*oooo}}{\dynkin[ply=2,backwards] E{ooo*oo}} & \myarrow \class  {\dynkin[ply=2,backwards] E{*o****}} &  & \lineno{2} \\
      \classeses{\dynkin[ply=2,backwards] E{*oooo*}}{\dynkin[ply=2,backwards] E{oo*o*o}} & \myarrow \class  {\dynkin[ply=2,backwards] E{o****o}} &  & \lineno{3} \\
      \classeses{\dynkin[ply=2,backwards] E{*oo*o*}}{\dynkin[ply=2,backwards] E{**ooo*}} & \mycircarrow                                         &  & \lineno{4} \\
      \class    {\dynkin[ply=2,backwards] E{oo***o}}                                    & \mycircarrow                                         &  & \lineno{5}
    \end{flalign*}
  \end{result}

  \begin{result}{$\Weyl(E_7)$}{}
    \begin{flalign*}
      \class  {\dynkin E{ooooooo}}                       & \myarrow \class  {\dynkin E{*******}}                     &  & \lineno{1} \\
      \classeses{\dynkin E{*oooooo}}{\dynkin E{oo*oooo}} & \myarrow \class  {\dynkin E{o******}}                     &  & \lineno{2} \\
      \classeses{\dynkin E{*oo*ooo}}{\dynkin E{*ooo*oo}} & \myarrow \class  {\dynkin E{o****o*}}                     &  & \lineno{3} \\
      \classeses{\dynkin E{*oo*o*o}}{\dynkin E{**oooo*}} & \myarrow \classes{\dynkin E{**oo*o*}}{\dynkin E{o**o*o*}} &  & \lineno{4} \\
      \class  {\dynkin E{o*oo*o*}}                       & \myarrow \class  {\dynkin E{o****oo}}                     &  & \lineno{5}
    \end{flalign*}
  \end{result}
  \vfill
  \newpage
  \mbox{}\vfill
  \myrule
  \begin{result}{$\Weyl(E_8)$}{}
    \begin{flalign*}
      \class    {\dynkin E{oooooooo}} & \myarrow \class {\dynkin E{********}}       &&    \lineno{1} \\
      \classeses{\dynkin E{*ooooooo}}{\dynkin E{oo*ooooo}} & \myarrow \class  {\dynkin E{*******o}}       &&    \lineno{2} \\
      \classeses{\dynkin E{*oo*oooo}}{\dynkin E{*ooo*ooo}} & \myarrow \class  {\dynkin E{o******o}}       &&    \lineno{3} \\
      \classeses{\dynkin E{*oo*o*oo}}{\dynkin E{**ooooo*}} & \myarrow \classes{\dynkin E{**oo*o*}}{\dynkin E{o**o*o*}}   &&    \lineno{4} \\
      \classeses{\dynkin E{*oo*o*o*}}{\dynkin E{**ooo*o*}} & \mycircarrow        &&    \lineno{5} \\
      \class  {\dynkin E{o****ooo}}  & \mycircarrow        &&    \lineno{6}
    \end{flalign*}
  \end{result}
  \vfill
  \newpage
  \mbox{}\vfill
  \myrule
  \begin{result}{$\Weyl(F_4)$}{}
    \begin{flalign*}
      \class {\dynkin F{oooo}}    & \myarrow \class  {\dynkin F{****}}   &&    \lineno{1} \\
      \classes {\dynkin F{*ooo}}{\dynkin F{o*oo}} & \myarrow \class  {\dynkin F{o***}}   &&    \lineno{2} \\
      \classes {\dynkin F{oo*o}}{\dynkin F{ooo*}} & \myarrow \class  {\dynkin F{***o}}   &&    \lineno{3} \\
      \classeses{\dynkin F{*o*o}}{\dynkin F{*oo*}} & \mycircarrow   &&    \lineno{4} \\
      \class {\dynkin F{o**o}}    & \mycircarrow   &&    \lineno{5}
    \end{flalign*}
  \end{result} 

  \begin{result}{$\Weyl(G_2(n))$}{(\(n\equiv 0 \mod 4\))}
    \begin{flalign*}
      \class    {\dynkin[gonality=n] G{oo}} & \myarrow \class  {\dynkin[gonality=n] G{**}}   &&    \lineno{1} \\
      \class    {\dynkin[gonality=n] G{*o}} & \myarrow \class  {\dynkin[gonality=n] G{o*}}   &&    \lineno{2}
    \end{flalign*}
  \end{result}

  \begin{result}{$\Weyl(G_2(n))$}{(\(n\equiv 2 \mod 4\))}
    \begin{flalign*}
      \class    {\dynkin[gonality=n] G{oo}} & \myarrow \class  {\dynkin[gonality=n] G{**}} &  & \lineno{1} \\
      \class    {\dynkin[gonality=n] G{*o}} & \mycircarrow                                 &  & \lineno{2} \\
      \class    {\dynkin[gonality=n] G{o*}} & \mycircarrow                                 &  & \lineno{3}
    \end{flalign*}
    An alternative common notation for \(G_2(n)\) is \(I_2(n)\).\\
    The Weyl group \(G_2\) is \(G_2 := G_2(6)\).
  \end{result}

  \begin{result}{$\Weylo(G_2(n))$}{(\(n\) odd)}
    \renewcommand{\arrowlabel}{\longest}
    \begin{flalign*}
      \class    {\dynkin[gonality=n] G{oo}} & \myarrow \class  {\dynkin[gonality=n] G{**}}   &&    \lineno{1}
    \end{flalign*}
    An alternative common notation for \(G_2(n)\) is \(I_2(n)\).
    Note that $\Weylo(G_2(n)) = \{\id,\longest\} \cong \Weyl(A_1)$.
  \end{result}

  \begin{result}{$\Weyl(H_3)$}{}
    \begin{flalign*}
       \class {\dynkin H{ooo}} &\myarrow \class {\dynkin H{***}} & & \lineno{1}\\
      \classeses {\dynkin H{*oo}}{\dynkin H{o*o}} & \myarrow \class {\dynkin H{*o*}} & & \lineno{2}
    \end{flalign*}
  \end{result}

  \begin{result}{$\Weyl(H_4)$}{}
    \begin{flalign*}
      \class {\dynkin H{oooo}} & \myarrow \class {\dynkin H{****}} && \lineno{1} \\
      \classeses {\dynkin H{*ooo}}{\dynkin H{o*oo}} & \myarrow \class {\dynkin H{***o}} && \lineno{2} \\
      \classeses {\dynkin H{*o*o}}{\dynkin H{*oo*}} & \mycircarrow&&\lineno{3}
    \end{flalign*}
  \end{result}
\end{@empty}
\vfill
\restoregeometry
\newpage
\bibliographystyle{alpha}
\bibliography{main}

\end{document}
